\documentclass[11pt]{article}

\usepackage[utf8]{inputenc}
\usepackage[british]{babel}
\usepackage[a4paper, margin=2.5cm]{geometry}
\usepackage{amsmath, amsthm, amssymb, amsfonts}
\usepackage{mathtools}
\usepackage{thmtools}
\usepackage[shortlabels]{enumitem}
\usepackage[font={small, it}]{caption}
\usepackage{microtype}
\usepackage{xcolor}
\usepackage{mathabx}
\usepackage{tikz}
\usepackage{hyperref}
\usepackage{bm}
\usepackage{comment}

\newcommand{\eps}{\varepsilon}
\renewcommand{\phi}{\varphi}
\renewcommand{\rho}{\varrho}

\let\setminus=\smallsetminus
\let\emptyset=\varnothing

\declaretheorem[parent=section]{theorem}
\declaretheorem[sibling=theorem]{lemma}

\setlist{itemsep=0.1em, topsep=0.1em, parsep=0.1em, partopsep=0.1em}

\hypersetup{
  colorlinks,
  linkcolor={red!60!black},
  citecolor={green!50!black},
  urlcolor={blue!80!black}
}

\colorlet{RoyalRed}{red!70!black}
\definecolor{RoyalBlue}{rgb}{0.25, 0.41, 0.88}
\definecolor{RoyalAzure}{rgb}{0.0, 0.22, 0.66}

\newlength{\bibitemsep}
\newlength{\bibparskip}
\let\oldthebibliography\thebibliography
\renewcommand\thebibliography[1]{%
  \oldthebibliography{#1}%
  \setlength{\parskip}{\bibitemsep}%
  \setlength{\itemsep}{\bibparskip}%
}

\title{Short proof of the hypergraph container theorem}
\author{
Rajko Nenadov\thanks{School of Computer Science, University of Auckland, New Zealand. Email: \texttt{rajko.nenadov@auckland.ac.nz}. Research supported by the New Zealand Marsden Fund.} \and
Huy Tuan Pham\thanks{Department of Mathematics, Stanford University, Stanford, CA 94305. Email: \texttt{huypham@stanford.edu.} Research supported by a Clay Research Fellowship and a Stanford Science Fellowship.}
}
\date{}

\begin{document}
\maketitle

\begin{abstract}
We present a short and simple proof of the celebrated hypergraph container theorem of Balogh--Morris--Samotij and Saxton--Thomason. On a high level, our argument utilises the idea of iteratively taking vertices of largest degree from an independent set and constructing a hypergraph of lower uniformity which preserves independent sets and inherits edge distribution. The original algorithms for constructing containers also remove in each step vertices of high degree which are not in the independent set. Our modified algorithm postpones this until the end, which surprisingly results in a significantly simplified analysis.
\end{abstract}

\section{Introduction}

The method of containers is a powerful technique in combinatorics used to produce a small number of clusters encompassing independent sets of a given hypergraph. While in some applications one follows the idea of the method and the general principles for building such clusters, quite often one can apply off the shelf tools. The most such applicable tool has been developed independently by Balogh, Morris, and Samorij \cite{balogh15container} and Saxton and Thomason \cite{saxton15containers}, and it is this result that is commonly referred to as the \emph{hypergraph container theorem}. For an introduction to the method, the hypergraph container theorem, and its many suprising applications, we refer the reader the ICM survey \cite{balogh18icm}. 
A number of different proofs and versions of this result have been obtained since \cite{balogh20efficient,bernshteyn19hyp, bfp, morris2018asymmetric,nenadov2023probabilistic, saxton16online, saxton16simple}. We present a simple and short proof of a slight generalisation of the original theorem. Two other short proofs have been obtained very recently by Campos and Samotij \cite{campos24short}.

Let $V$ be a finite set. Given a subset $X \subseteq V$, let $\langle X \rangle = \{S \subseteq V \colon X \subseteq S\}$. We say a probability measure $\nu$ over $2^{V}$ is \emph{$(p, K)$-uniformly-spread} if for every non-empty $X \subseteq V$ we have $\nu(\langle X \rangle) \le K p^{|X|-1} / |V|$. \emph{Uniform} signifies that the measure is fairly uniform from the point of view of elements of $V$. Throughout the paper we use $V = V(\mathcal{H})$ and $N = |V|$, where $\mathcal{H}$ is a given hypergraph. If all edges in a hypergraph $\mathcal{H}$ have size at most $\ell$, we say that $\mathcal{H}$ is an \emph{$(\le \ell)$-graph}.

\begin{theorem} \label{thm:main}
    For every $\ell \in \mathbb{N}$ and $K, \eps > 0$ there exists $T > 0$ such that the following holds. Suppose $\mathcal{H}$ is an $(\le \ell)$-graph, and let $\nu$ be $(p, K)$-uniformly-spread measure over $2^{V}$ supported on $\mathcal{H}$, for some $p \in (0,1]$. Then for every independent set $I \subseteq V(\mathcal{H})$ there exists $F \subseteq I$ and $C = C(F) \subseteq V$ such that $|F| \le T N p$, $\nu(\mathcal{H}[C]) < \eps$, and $I \subseteq C \cup F$.
\end{theorem}

If $\nu$ is uniform on $\mathcal{H}$, we obtain original hypergraph container theorems \cite{balogh15container, saxton15containers}. Dependence of $T$ on the uniformity is of order $O(2^{\ell^2})$, which is also along the lines of the original results. Near-optimal dependence was obtained by Balogh and Samotij \cite{balogh20efficient} and Campos and Samotij \cite{campos24short}.

\section{Proof}

Our proof bears resemblance with the proof from \cite{balogh15container,saxton15containers}. On a high level, we choose $F$ in Theorem \ref{thm:main} by greedily taking vertices from $I$ with largest degree with respect to $\nu$ and construct a hypergraph of lower uniformity given by (parts) of hyperedges containing vertices from $F$. A common feature in many of the proofs utilising a similar idea is that one also keeps track of the vertices which are not in $I$ but have larger degree than the last chosen vertex in $F$. The main novelty here is that we completely avoid this, unless we are in the case the resulting hypergraph of lower uniformity is not sufficiently dense to proceed with the induction. In this case, we show that removing vertices of high degree immediately yields a desired container. 
It is worth noting that the proofs from \cite{balogh15container,saxton15containers} also have a similar case distinction, however the analysis in our cases turns out to be significantly simpler.

Theorem \ref{thm:main} follows by iterated application of the following lemma, known as the \emph{hypergraph container} lemma.

\begin{lemma} \label{lemma:main}
    For every $\ell \in \mathbb{N}$ and $K > 0$ there exists $\delta > 0$ such that the following holds. Suppose $\mathcal{H}$ is an $(\le \ell)$-graph, and let $\nu$ be a $(p, K)$-uniformly-spread measure over $2^{V}$ supported on $\mathcal{H}$, for some $p \in (0, 1]$. Then for every independent set $I \subseteq V$ there exists $F \subseteq I$ and $C = C(F) \subseteq V$ such that $|F| \le \ell N p$, $|C| \le (1 - \delta) N$, and $I \subseteq C \cup F$. Moreover, $C$ can be unambigously constructed from any $F \subseteq \hat F \subseteq I$.
\end{lemma}
\begin{proof}
    We prove the lemma by induction on $\ell$. For $\ell = 1$, take $F = \emptyset$ and $C \subseteq V$ to be the set of all vertices $v \in V$ with $\nu(v) = 0$. As there are at least $N / K$ vertices with strictly positive measure, the lemma holds for $\delta = 1 / K$. We now prove the lemma for $\ell \ge 2$. Without loss of generality, we may assume $|I| \ge Np$.

    Set $F = \emptyset \subseteq I$, $\mathcal{L} = \emptyset \subseteq 2^{V}$, and $\mathcal{D}, \mathcal{H}' = \emptyset \subseteq \mathcal{H}$. Repeat the following for $Np$ rounds: Take $v \in I \setminus F$ to be a largest vertex with respect to $\nu(\langle v \rangle \cap \mathcal{R})$, where $\mathcal{R} = \mathcal{H}[V \setminus F] \setminus \mathcal{D}$ (tie-breaking done in some canonical way, e.g.\ by agreeing on the ordering of $V$). Add $v$ to $F$, set $\mathcal{H}' = \mathcal{H}' \cup (\langle v \rangle \cap \mathcal{R})$, and for each $X \in  2^{V} \setminus \mathcal{L}$ of size $|X| \le \ell -1$ such that
    \begin{equation} \label{eq:violated}
        \nu(\langle X \rangle \cap \mathcal{H}') > K p^{|X|} / N,
    \end{equation}
    add $X$ to $\mathcal{L}$ and set $\mathcal{D} = \mathcal{D} \cup (\langle X \rangle \cap \mathcal{R})$. 
    
    A few observations about the process. First, as $\nu$ is $(p, K)$-uniformly-spread the value $\nu(\langle X \rangle \cap \mathcal{H}')$ increases by at most $\nu(\langle X \cup \{v\} \rangle) \le Kp^{|X|} / N$ after adding a vertex $v$ to $F$. Once a subset $X$ satisfies \eqref{eq:violated} no more hyperedges which contain $X$ are added to $\mathcal{H}'$, thus at the end of the process we have
    \begin{equation} \label{eq:upper_bound}
        \nu(\langle X \rangle \cap \mathcal{H}') \le 2 K p^{|X|} / N
    \end{equation}
    for every $X \subseteq V$ of size $|X| \le \ell - 1$. Second, given $F \subseteq \hat F \subseteq I$, we can reconstruct $F$ from $\hat F$ together with the order in which the vertices were added, thus we can also reconstruct $\mathcal{H'}$ and $\mathcal{R}$.

    We next derive several useful lower bounds on $\nu(\mathcal{H}')$. First we show that if $\nu(\mathcal{D})$ is large, then $\nu(\mathcal{H}')$ is also large. In particular, the following holds:
    \begin{align}\label{eq:D-large}
    	\nu(\mathcal{H}') \ge 2^{-\ell} p\nu(\mathcal{D}).
    \end{align}
    Indeed, for each $e \in \mathcal{D}$ there exists $X \in \mathcal{L}$ such that $e \in \langle X \rangle$. Thus, $\sum_{X\in \mathcal{L}} \nu(\langle X\rangle) \ge \nu(\mathcal{D})$. On the other hand, we have by (\ref{eq:violated}) that
    \begin{equation*}
    	\sum_{X\in \mathcal{L}} \nu(\langle X\rangle \cap \mathcal{H}') > \sum_{X\in \mathcal{L}} Kp^{|X|}/N \ge  p \sum_{X\in \mathcal{L}} \nu(\langle X\rangle).
    \end{equation*}
    Here in the last inequality we use that $\nu$ is $(p,K)$-uniformly spread. Furthermore, each edge $e$ in $\mathcal{H}'$ may contribute to at most $2^{\ell}$ terms $\nu(\langle X\rangle \cap \mathcal{H}')$. Hence, 
    \begin{align*}
    	\nu(\mathcal{H}') \ge 2^{-\ell} \sum_{X\in \mathcal{L}} \nu(\langle X\rangle \cap \mathcal{H}') > 2^{-\ell} p\nu(\mathcal{D}),
    \end{align*}
    as claimed in (\ref{eq:D-large}). 
    
    Next, we show that 
    \begin{align}\label{eq:v-upp}
    	\nu(\mathcal{H}') \ge (Np) \max_{v\in I\setminus F} \nu(\langle v \rangle \cap \mathcal{R}). 
    \end{align}
    Let $\mathcal{R}_i$ denote the hypergraph $\mathcal{R}$ at the moment when the $i$-th vertex $v_i$ was added to $F$ (thus $\mathcal{R} = \mathcal{R}_{|F|}$). We observe that, since $\mathcal{R}$ is non-increasing and by our choice of $v$ in each step, 
    \begin{align*}
    	\nu(\mathcal{H}') \ge \sum_{i=1}^{|F|} \nu(\langle v_i\rangle \cap \mathcal{R}_i) \ge  \sum_{i=1}^{|F|} \max_{v\in I\setminus F} \nu(\langle v \rangle \cap \mathcal{R}_{|F|}),
    \end{align*}
    yielding (\ref{eq:v-upp}). 
    
    Let $\alpha = 2^{-\ell-2}$. We now distinguish two cases, where if $\nu(\mathcal{H}')$ is large, then we can apply the inductive hypothesis to an appropriate $(\le \ell-1)$-graph, and otherwise we can immediately find a small container $C$ for which $I\setminus F \subseteq C$.  
    
     \textbf{Case 1: $\nu(\mathcal{H}') \ge \alpha p$.} Let $\mathcal{H}''$ denote the $(\le \ell-1)$-graph consisting of sets $X$ such that $X=H'\setminus F$ for some $H'\in \mathcal{H}'$. Set $\nu'$ to be the probability measure over $2^{V \setminus F}$ given by
    $$
        \nu'(X) \, \propto\,  \begin{cases}
            \nu( (X \cup 2^F)  \cap \mathcal{H}'), &\text{if } X\in \mathcal{H}'', \\
            0, &\text{otherwise.}
        \end{cases}
    $$
    From \eqref{eq:upper_bound} and $\nu(\mathcal{H}') \ge \alpha p$ we conclude that $\nu'$ is $(2K \alpha^{-1}, p)$-uniformly-spread. 
    Also observe that $I$ is an independent set in $\mathcal{H}''$, thus by the induction hypothesis there exists $F' \subseteq V$ of size $|F'| \le (\ell - 1)Np$ and $C = C(F')$ such that $|C| \le (1 - \delta) N$ and $I \subseteq C \cup F'$. Note that we can reconstruct $C$ from $F := F \cup F'$.
    
    \textbf{Case 2: $\nu(\mathcal{H}') < \alpha p$.} By (\ref{eq:D-large}), we have $\nu(\mathcal{D})<1/4$ and hence $\nu(\mathcal{R}) \ge \nu(\mathcal{H})-\nu(\mathcal{H}')-\nu(\mathcal{D}) > 1/2$. By (\ref{eq:v-upp}), for every $v\in I\setminus F$ we have 
    \begin{equation} \label{eq:v_upper}
        \nu(\langle v \rangle \cap \mathcal{R}) \le \alpha / N.
    \end{equation}
    Let now $C \subseteq V \setminus F$ denote the set of all $v \in V \setminus F$ such that $\nu(\langle v \rangle \cap \mathcal{R}) \le \alpha / N$. By \eqref{eq:v_upper} we have $I \setminus F \subseteq C$. Furthermore, 
    \[
        \nu(\mathcal{R}) \le \sum_{v \in C} \nu(\langle v \rangle \cap \mathcal{R}) + \sum_{w \in V \setminus (F \cup C)} \nu(\langle w \rangle \cap \mathcal{R}) < \alpha + (N-|C|) \cdot K/N.
    \]
    Hence, $|C| < N - (\nu(\mathcal{R})-\alpha)N/K < (1-\delta)N$ for $\delta=1/(4K)$. This concludes the construction of desired $F$ and $C$.
\end{proof}

For the sake of completeness, we derive Theorem \ref{thm:main} from Lemma \ref{lemma:main}.

\begin{proof}[Proof of Theorem \ref{thm:main}]
    Let $\delta > 0$ be as given by Lemma \ref{lemma:main} for $\ell$ and $K/\eps$ (as $K$). We prove the theorem for $T = \ell \log(K \eps^{-1}) / \log(1 + \delta)$.
    
    We find a \emph{fingerprint} $F$ and a \emph{container} $C$ as follows. Set $F = \emptyset$ and $C = V$, and as long as $\nu(\mathcal{H}[C]) \ge \eps$ do the following: Let $F'$ and $C'$ be as given by Lemma \ref{lemma:main} applied with $\nu'$ being a probability measure over $2^C$ given by $\nu'(X) \, \propto \, \nu(X)$ if $X \in \mathcal{H}[C]$, and $\nu'(X) = 0$ otherwise. Set $F := F \cup F'$ and $C := C'$, and proceed to the next iteration.

    If $\nu(\mathcal{H}[C]) \ge \eps$, 
    then for nonempty $X\subseteq C$, 
    \[
        \nu'(\langle X\rangle) \le \frac{\nu(\langle X\rangle)}{\nu(\mathcal{H}[C])} \le \frac{Kp^{|X|-1}/N}{\eps} \le \frac{K}{\eps} p^{|X|-1} / |C|,
    \]
    and hence $\nu'$ is $(p, K/\eps)$-uniformly-spread each time we apply Lemma \ref{lemma:main}. Furthermore, if $\nu(\mathcal{H}[C]) \ge \eps$, then $|C| \ge \eps N / K$. In each iteration the set $C$ shrinks by a factor of $1 - \delta$, thus we are done after at most $\log(K \eps^{-1}) / \log(1 + \delta)$ iterations. The set $F$ grows by at most $\ell Np$ in each iteration, which gives an upper bound of $T Np$ on its final size for the above choice of $T = T(K, \eps)$. Due to the last property in Lemma \ref{lemma:main}, the final set $C$ can be unambiguously constructed from $F$.
\end{proof}

\paragraph{Acknowledgment.} Ideas used in this paper were developed while the first author was visiting Stanford University in November 2023. The first author thanks Jacob Fox for hospitality. We thank Jacob Fox for helpful comments and Wojciech Samotij for pointing out a subtle issue in an earlier version.

\bibliographystyle{abbrv}
\bibliography{hypergraph_containers}

\end{document}